\documentclass[a4paper, twoside,11pt]{article}
\usepackage{fancyhdr}
\usepackage{fnpos}
 \usepackage[english]{babel}        
\usepackage[T1]{fontenc}          
\usepackage{graphicx}             
\usepackage{makeidx}
\usepackage{fancybox}
\usepackage{framed}
\usepackage{fancyhdr}
 \usepackage{pstricks,pst-plot,pstricks-add}
\usepackage[margin=1in]{geometry}
\usepackage{graphicx}
\usepackage{titlesec}
\usepackage{amsmath}
\usepackage{amsfonts}   
\usepackage{amssymb}    
\usepackage{amsthm}
\usepackage{dsfont}
\usepackage{mathtools}
\usepackage{verbatim}   
\usepackage{color}
\usepackage{listings}
\usepackage{graphicx}
\usepackage{amsmath}
\usepackage{amsmath,amssymb,color,inputenc,euscript,graphicx,psfrag}

\DeclareMathOperator*{\esssup}{ess\,sup}

\usepackage{amsmath, amsthm, amscd, amsfonts, amssymb, graphicx, color,   mathrsfs}
\usepackage[english]{babel}

\newtheorem{thm}{Theorem}[section]
\newtheorem{cor}[thm]{Corollary}
\newtheorem{lem}[thm]{Lemma}
\newtheorem{prop}[thm]{Proposition}
\newtheorem{defn}[thm]{Definition}
\newtheorem{rem}[thm]{Remark}
\numberwithin{equation}{section}

\renewcommand{\thefootnote}

\renewcommand\Im{\operatorname{Im}}

 {\markboth{\sl A. Saoudi }{\sl Time-frequency Analysis of two-wavelet theory in Weinstein setting}

\newcommand{\X}{\mathcal{X}}

\author {Ahmed Saoudi}

\title{Time-frequency Analysis of two-wavelet theory in Weinstein setting}

\date{}
 
\begin{document}
 \maketitle
\begin{center}
     Northern Border University, College of Science, Arar, P.O. Box 1631, Saudi Arabia.\\
   Universit\'{e} de Tunis El Manar, Facult\'{e} des sciences de Tunis, Tunisie.\\
     \textbf{ e-mail:} ahmed.saoudi@ipeim.rnu.tn
\end{center}
  \begin{abstract}
In this paper, we introduce the notion of  Weinstein two-wavelet and we define the two-wavelet localization operators in the setting of the Weinstein theory. Then we  give a host of sufficient conditions for the boundedness and compactness of the two-wavelet localization operator  on $L^{p}_{\alpha}(\mathbb{R}^{d+1}_+)$ for all $1\leq p\leq \infty$, in terms of properties of the symbol $\sigma$ and the functions $\varphi$ and $\psi$. In the end, we study some typical examples of the Weinstein two-wavelet localization operators.

 \end{abstract}
  \section{ Introduction}
  The Weinstein operator $\Delta_{W,\alpha}^d$ defined on $\mathbb{R}_{+}^{d+1}=\mathbb{R}^d\times(0, \infty)$, by
\begin{equation*}
\Delta_{W,\alpha}^d=\sum_{j=1}^{d+1}\frac{\partial^2}{\partial x_j^2}+\frac{2\alpha+1}{x_{d+1}}\frac{\partial}{\partial x_{d+1}}=\Delta_d+L_\alpha,\;\alpha>-1/2,
\end{equation*}
where $\Delta_d$ is the Laplacian operator for the $d$ first variables and $L_\alpha$ is the Bessel operator for the last variable defined on $(0,\infty)$ by
$$L_\alpha u=\frac{\partial^2 u}{\partial x_{d+1}^2}+\frac{2\alpha+1}{x_{d+1}}\frac{\partial u}{\partial x_{d+1}}.$$
The Weinstein operator $\Delta_{W,\alpha}^d$ has several applications in pure and applied mathematics, especially in fluid mechanics \cite{brelot1978equation, weinstein1962singular}.

Very recently, many authors have been investigating the behaviour of the Weinstein
transform (\ref{defWeinstein}) with respect to several problems already studied for the classical Fourier transform.
For instance,   Heisenberg-type inequalities \cite{salem2015heisenberg}, Littlewood-Paley g-function \cite{salem2016littlewood},
Shapiro and Hardy–Littlewood–Sobolev type inequalities \cite{salem2020hardy, salem2015shapiro},
 Paley-Wiener theorem \cite{mehrez2017paley}, Uncertainty principles \cite{mejjaoli2011uncertainty, ahmed2018variation, saoudi2019l2}, multiplier Weinstein operator \cite{ahmed2018calder}, wavelet and continuous wavelet transform \cite{gasmi2016inversion, mejjaoli2017new}, Wigner transform and localization operators \cite{saoudi2019weinstein, saoudilocalisation}, and so forth...

In the classical setting, the notion of wavelets was first introduced by Morlet in connection with his study of seismic traces and the mathematical foundations were given by Grossmann and Morlet \cite{grossmann1984decomposition}. Later, Meyer and many other mathematicians recognized many classical results of this theory \cite{koornwinder1993continuous, meyer1992wavelets}. Classical wavelets have wide applications, ranging from signal analysis in geophysics and acoustics to quantum theory and pure mathematics \cite{daubechies1992ten, goupillaud1984cycle, holschneider1995wavelets}.

Recently, the theory of wavelets and continuous wavelet transform has been extended and generalized in the context of differential-difference operators \cite{gasmi2016inversion, mejjaoli2017dunkl, mejjaoli2017new,  mejjaoli2017time}.
Wavelet analysis has attracted attention for its ability to analyze rapidly changing transient signals. Any application using the Fourier like transform can be formulated using wavelets to provide more accurately localized temporal
and frequency information. The reason for the extension from one wavelet to two wavelets comes from the extra
degree of flexibility in signal analysis and imaging when the localization operators are
used as time-varying filters. One of the aims of the continuous wavelet transform, is the study of their localization
operators.

The time-frequency representations required for localization operators wish  have been object of study in quantum mechanics, in PDE and signal analysis recently. In engineering, a natural language is given by time-frequency analysis. Localization operators arise from pure and applied mathematics in connection with various areas of research. They were initiated by Daubechies \cite{daubechies1988time, daubechies1990wavelet, daubechies1988time2}, and before she highlighted the role of these operators to localize a signal simultaneously in time and frequency.

Nowadays, these operators have found many applications to time-frequency analysis, the theory of differential equations, quantum mechanics. Depending on the field of application, these operators are known under the names of Wick, anti-Wick or Toeplitz operators, as well as wave packets, Gabor or short time Fourier transform multipliers. Arguing from these point of view, many works were done on them,  we refer, for instance \cite{balazs2008hilbert, balazs2012multipliers, cordero2003time, de2002uniform, grochenig2013foundations, mejjaoli2017dunkl}.

Using the harmonic analysis associated with the Weinstein operator (generalized translation operators,
generalized convolution, Weinstein transform, ...) and the same idea as for the classical case,
we study the localisations operators associated with the Weinstein two-wavlet \cite{saoudi2020two} and we prove that under
suitable conditions on the symbols and two Weinstein wavelets, the boundedness and
compactness of these localization operators. Our main results for the boundedness and compactness of the Weinstein two wavelet localisation operators, with different symbols and windows, are summarized in the following table.
\begin{table}[h]
\begin{center}
\begin{tabular}{| c | c | c |c |}
\hline
\textbf{Symbol} &  \multicolumn{2}{|c|}{\textbf{Windows}} & \textbf{Localization Operator}  \\ \hline
$\sigma$ & $\varphi$ & $\psi$ & $\mathcal{L}_{\varphi,\psi}(\sigma)$ \\ \hline \hline
$L^1_\alpha(\X)$ & $L^\infty_\alpha(\mathbb{R}^{d+1}_+)$ & $L^1_\alpha(\mathbb{R}^{d+1}_+)$ &
$\mathcal{B}( L^1_\alpha(\mathbb{R}^{d+1}_+))$ \\ \hline
$L^1_\alpha(\X)$ & $L^1_\alpha(\mathbb{R}^{d+1}_+)$ & $L^\infty_\alpha(\mathbb{R}^{d+1}_+)$ & $\mathcal{B}( L^\infty_\alpha(\mathbb{R}^{d+1}_+))$ \\ \hline
$L^1_\alpha(\X)$  & \multicolumn{2}{|c|}{$L^1_\alpha\bigcap L^\infty_\alpha(\mathbb{R}^{d+1}_+)$} & $\mathcal{B}( L^p_\alpha(\mathbb{R}^{d+1}_+)),\, p\in [1,\infty]$ \\ \hline
$L^1_\alpha(\X)$ & $L^q_\alpha(\mathbb{R}^{d+1}_+)$ & $L^p_\alpha(\mathbb{R}^{d+1}_+)$ & $\mathcal{B}( L^p_\alpha(\mathbb{R}^{d+1}_+)),\, p\in [1,\infty]$ \\ \hline
$L^1_\alpha(\X),\,  r\in [1,2]$  & \multicolumn{2}{|c|}{$L^1_\alpha\bigcap L^2_\alpha\bigcap L^\infty_\alpha(\mathbb{R}^{d+1}_+)$} & $\mathcal{B}( L^p_\alpha(\mathbb{R}^{d+1}_+)),\, p\in [r,r']$ \\ \hline
\end{tabular}
\end{center}
\caption{Boundedness and compactness of localisation operators}
\end{table}

This paper is organized as follows. In Section 2, we recall some properties of harmonic
analysis for the Weinstein operators and Weinstein two-wavelet theory. In Section 3,  we  give a host of sufficient conditions for the boundedness and compactness of the two-wavelet localization operator  on $L^{p}_{\alpha}(\mathbb{R}^{d+1}_+)$ for all $1\leq p\leq \infty$, in terms of properties of the symbol $\sigma$ and the functions $\varphi$ and $\psi$. In the end, we study some typical examples of the Weinstein two-wavelet localization operators.
 \section{Preliminaires}
\subsection{Harmonic analysis associated with the Weinstein operator}

For all $\lambda=(\lambda_1,...,\lambda_{d+1})\in\mathbb{C}^{d+1}$, the system
\begin{equation}
\begin{gathered}
\frac{\partial^2u}{\partial x_{j}^2}(  x)
  =-\lambda_{j} ^2u(x), \quad\text{if } 1\leq j\leq d \\
L_{\alpha}u(  x)  =-\lambda_{d+1}^2u(  x), \\
u(  0)  =1, \quad \frac{\partial u}{\partial
x_{d+1}}(0)=0,\quad \frac{\partial u}{\partial
x_{j}}(0)=-i\lambda_{j}, \quad \text{if } 1\leq j\leq d
\end{gathered}
\end{equation}
 has a unique solution  denoted by $\Lambda_{\alpha}^d(\lambda,.),$ and given by
\begin{equation}\label{wkernel}
\Lambda_{\alpha}^d(\lambda,x)=e^{-i<x^\prime,\lambda^\prime>}j_\alpha(x_{d+1}\lambda_{d+1})
\end{equation}
 where $x=(x^\prime,x_{d+1}),\; x_d'=(x_1,x_2,\cdots,x_d),\; \lambda=(\lambda^\prime,\lambda_{d+1}) ,\; \lambda_d'=(\lambda_1,\lambda_2,\cdots,\lambda_d)$ and $j_\alpha$ is the normalized Bessel function of index $\alpha$ defined by
$$j_\alpha(x)=\Gamma(\alpha+1)\sum_{k=0}^\infty\frac{(-1)^k x^{2k}}{2^k k!\Gamma(\alpha+k+1)}.$$
The function $(\lambda,x)\mapsto\Lambda_{\alpha}^d(\lambda,x)$ is called the Weinstein kernel and has a unique extension to $\mathbb{C}^{d+1}\times\mathbb{C}^{d+1}$, and satisfied the following properties.\\
\begin{itemize}
\item[(i)] For all $(\lambda,x)\in \mathbb{C}^{d+1}\times\mathbb{C}^{d+1}$ we have
\begin{equation*}
\Lambda_{\alpha}^d(\lambda,x)=\Lambda_{\alpha}^d(x,\lambda).
\end{equation*}
\item[(ii)] For all $(\lambda,x)\in \mathbb{C}^{d+1}\times\mathbb{C}^{d+1}$ we have
\begin{equation*}
\Lambda_{\alpha}^d(\lambda,-x)=\Lambda_{\alpha}^d(-\lambda,x).
\end{equation*}
\item[(iii)] For all $(\lambda,x)\in \mathbb{C}^{d+1}\times\mathbb{C}^{d+1}$ we get
\begin{equation*}
\Lambda_{\alpha}^d(\lambda,0)=1.
\end{equation*}
\item[(iv)] For all $\nu\in\mathbb{N}^{d+1},\;x\in\mathbb{R}^{d+1}$ and $\lambda\in\mathbb{C}^{d+1}$ we have
\begin{equation*}\label{klk}
 \left|D_\lambda^\nu\Lambda_{\alpha}^d(\lambda,x)\right|\leq\left\|x\right\|^{\left|\nu\right|}e^{\left\|x\right\|\left\|\Im \lambda\right\|}
\end{equation*}
\end{itemize}
where $D_\lambda^\nu=\partial^\nu/(\partial\lambda_1^{\nu_1}...\partial\lambda_{d+1}^{\nu_{d+1}})$ and $\left|\nu\right|=\nu_1+...+\nu_{d+1}.$ In particular, for all $(\lambda,x)\in \mathbb{R}^{d+1}\times\mathbb{R}^{d+1}$, we have
\begin{equation}\label{normLambda}
\left|\Lambda_{\alpha}^d(\lambda,x)\right|\leq 1.
\end{equation}

In the following we denote by
\begin{itemize}
\item [(i)] $-\lambda=(-\lambda',\lambda_{d+1})$
\item[(ii)] $C_*(\mathbb{R}^{d+1})$, the space of continuous functions on $\mathbb{R}^{d+1},$ even with respect to the last variable.
    \item[(iii)] $S_*(\mathbb{R}^{d+1})$, the space of the $C^\infty$ functions, even with respect to the last variable, and rapidly decreasing together with their derivatives.
  \item[(iv)] $\mathcal{S_*}(\mathbb{R}^{d+1}\times\mathbb{R}^{d+1})$, the Schwartz space of rapidly decreasing functions on $\mathbb{R}^{d+1}\times\mathbb{R}^{d+1}$ even with respect to the last two variables.
  \item[(v)]  $\mathcal{D_*}(\mathbb{R}^{d+1})$, the space of $C^\infty$-functions on $\mathbb{R}^{d+1}$ which are of compact support,even with respect to the last variable.
	\item[(vi)] $L^p_\alpha(\mathbb{R}^{d+1}_+),\;1\leq p\leq \infty,$ the space of measurable functions $f$ on $\mathbb{R}^{d+1}_+$ such that
	$$\left\|f\right\|_{\alpha,p}=\left(\int_{\mathbb{R}^{d+1}_+}\left|f(x)\right|^pd\mu_\alpha(x)\right)^{1/p}<\infty, \;p\in[1,\infty),$$
	$$\left\|f\right\|_{\alpha,\infty}=\textrm{ess}\sup_{x\in\mathbb{R}^{d+1}_+}\left|f(x)\right|<\infty,$$

where $d\mu_{\alpha}(x)$ is the measure on  $\mathbb{R}_{+}^{d+1}=\mathbb{R}^d\times(0,\infty)$ given by
\begin{equation*}\label{mesure}
	d\mu_\alpha(x)=\frac{x^{2\alpha+1}_{d+1}}{(2\pi)^d2^{2\alpha}\Gamma^2(\alpha+1)}dx.
	\end{equation*}
\end{itemize}

For a radial function $\varphi\in L_{\alpha}^{1}(\mathbb{R}_{+} ^{d+1})$ the function $\tilde{\varphi}$ defined on $\mathbb{R}_+$ such that $\varphi(x)=\tilde{\varphi}(|x|)$, for all
$x\in\mathbb{R}_{+} ^{d+1}$, is integrable with respect to the measure $r^{2\alpha+d+1}dr$, and we have
\begin{equation}\label{radialweinstein}
  \int_{\mathbb{R}_{+}^{d+1}}\varphi(x)d\mu_{\alpha}(x)=a_\alpha\int_{0}^{\infty}
  \tilde{\varphi}(r)r^{2\alpha+d+1}dr,
\end{equation}
where $$a_\alpha=\frac{1}{2^{\alpha+\frac{d}{2}}\Gamma(\alpha+\frac{d}{2}+1)}.$$
The Weinstein transform generalizing the usual Fourier transform, is given for
$\varphi\in L_{\alpha}^{1}(\mathbb{R}_{+} ^{d+1})$ and $\lambda\in\mathbb{R}_{+}^{d+1}$, by

\begin{equation}\label{defWeinstein}
\mathcal{F}_{W}
(\varphi)(\lambda)=\int_{\mathbb{R}_{+}^{d+1}}\varphi(x)\Lambda_{\alpha}^d(x, \lambda
)d\mu_{\alpha}(x),
\end{equation}

We list some known basic properties of the Weinstein transform are as follows. For the proofs, we refer \cite{nahia1996spherical, nahia1996mean}.

\begin{itemize}
	\item[(i)] For all $\varphi\in L^1_\alpha(\mathbb{R}^{d+1}_+)$, the function $\mathcal{F}_{W}(\varphi)$  is continuous on $\mathbb{R}^{d+1}_+$ and we have
	\begin{equation}\label{L1-Linfty}
	\left\|\mathcal{F}_{W}\varphi\right\|_{\alpha,\infty}\leq\left\|\varphi\right\|_{\alpha,1}.
	\end{equation}
	\item[(ii)]   The Weinstein transform is a topological isomorphism from $\mathcal{S}_*(\mathbb{R}^{d+1})$ onto itself. The inverse transform is given by
	\begin{equation}\label{inversionweinstein}
	\mathcal{F}_{W}^{-1}\varphi(\lambda)= \mathcal{F}_{W}\varphi(-\lambda),\;\textrm{for\;all}\;\lambda\in\mathbb{R}^{d+1}_+.
	\end{equation}
	\item[(iii)] For all $f$ in $\mathcal{D_*}(\mathbb{R}^{d+1})$ (resp. $\mathcal{S_*}(\mathbb{R}^{d+1})$),  we have the following relations
\begin{equation}\label{fourierbar}
 \forall\lambda\in\mathbb{R}^{d+1}_+,\quad \mathcal{F}_{W}(\overline{\varphi})(\lambda)= \overline{\mathcal{F}_{W}(\widetilde{\varphi})(\lambda)},
\end{equation}
\begin{equation}\label{fourierbar1}
 \forall\lambda\in\mathbb{R}^{d+1}_+,\quad \mathcal{F}_{W}(\varphi)(\lambda)= \mathcal{F}_{W}(\widetilde{\varphi})(-\lambda),
\end{equation}
where $\widetilde{\varphi}$ is the function defined by
\begin{equation*}
  \forall\lambda\in\mathbb{R}^{d+1}_+,\quad\widetilde{\varphi}(\lambda)=\varphi(-\lambda).
\end{equation*}
	\item[(iv)] Parseval's formula: For all $\varphi, \phi\in \mathcal{S}_*(\mathbb{R}^{d+1})$, we have
	\begin{equation}\label{MM} \int_{\mathbb{R}^{d+1}_+}\varphi(x)\overline{\phi(x)}d\mu_\alpha(x)=\int_{\mathbb{R}^{d+1}_+}\mathcal{F}_{W}
(\varphi)(x)\overline{\mathcal{F}_{W}(\phi)(x)}d\mu_\alpha(x).
	\end{equation}
\item[(v)] Plancherel's formula: For all $\varphi\in L^2_\alpha(\mathbb{R}^{d+1}_+)$, we have
\begin{equation}\label{Plancherel formula}
	\left\|\mathcal{F}_{W}\varphi\right\|_{\alpha,2}=\left\|\varphi\right\|_{\alpha,2}.
	\end{equation}
\item[(vi)] Plancherel Theorem: The Weinstein transform $\mathcal{F}_{W}$ extends uniquely to an isometric isomorphism on $L^2_\alpha(\mathbb{R}^{d+1}_+).$
\item[(vii)] Inversion formula: Let $\varphi\in L^1_\alpha(\mathbb{R}^{d+1}_+)$ such that $\mathcal{F}_{W}\varphi\in L^1_\alpha(\mathbb{R}^{d+1}_+)$,  then we have
\begin{equation}\label{inv}
\varphi(\lambda)=\int_{\mathbb{R}^{d+1}_+}\mathcal{F}_{W}\varphi(x)\Lambda_{\alpha}^d(-\lambda,x)d\mu_\alpha(x),\;\textrm{a.e. }\lambda\in\mathbb{R}^{d+1}_+.
\end{equation}
\end{itemize}

Using relations (\ref{L1-Linfty}) and (\ref{Plancherel formula}) with Marcinkiewicz's interpolation theorem \cite{zbMATH03367521} we deduce that for every $\varphi\in L^p_\alpha(\mathbb{R}^{d+1}_+)$ for all $1\leq p\leq 2$, the function $\mathcal{F}_{W}(\varphi)\in L^q_\alpha(\mathbb{R}^{d+1}_+), q=p/(p-1),$ and
\begin{equation}\label{Lp-Lq}
	\left\|\mathcal{F}_{W}\varphi\right\|_{\alpha,q}\leq\left\|\varphi\right\|_{\alpha,p}.
	\end{equation}

\begin{defn} The translation operator $\tau^\alpha_x,\;x\in\mathbb{R}^{d+1}_+$ associated with the Weinstein operator $\Delta_{W,\alpha}^d$, is defined for a continuous function $\varphi$ on $\mathbb{R}^{d+1}_+$, which is even with respect to the last variable and for all $y\in\mathbb{R}^{d+1}_+$ by
$$\tau^\alpha_x\varphi(y)=C_\alpha\int_0^\pi\varphi\left(x^\prime+y\prime,\sqrt{x^2_{d+1}+y^2_{d+1}+2x_{d+1}y_{d+1}
\cos\theta}\right)\left(\sin\theta\right)^{2\alpha}d\theta,$$
with $$C_\alpha=\frac{\Gamma(\alpha+1)}{\sqrt{\pi}\Gamma(\alpha+1/2)}.$$
\end{defn}
By using the Weinstein kernel, we can also define a generalized translation, for
a function $\varphi\in\mathcal{S}_*(\mathbb{R}^{d+1})$ and $y\in\mathbb{R}^{d+1}_+$ the generalized translation $\tau^\alpha_x\varphi$ is defined by the following relation
\begin{equation}\label{MMM}
\mathcal{F}_{W}(\tau^\alpha_x\varphi)(y)=\Lambda^d_\alpha(x,y)\mathcal{F}_{W}(\varphi)(y).
\end{equation}

In the following proposition, we give some properties of the Weinstein
translation operator:

\begin{prop} The translation operator $\tau^\alpha_x,\;x\in\mathbb{R}^{d+1}_+$ satisfies the following properties.\\
i). For $\varphi\in\mathbb{C}_*(\mathbb{R}^{d+1})$, we have for all $x,y\in\mathbb{R}^{d+1}_+$
\begin{equation}\label{symtrictrans}
\tau^\alpha_x\varphi(y)=\tau^\alpha_y\varphi(x)\;\textrm{and}\;\tau^\alpha_0\varphi=\varphi.
\end{equation}
ii). Let $\varphi\in L^p_\alpha(\mathbb{R}^{d+1}_+),\;1\leq p\leq \infty$ and $x\in\mathbb{R}^{d+1}_+$. Then  $\tau^\alpha_x\varphi$ belongs to $L^p_\alpha(\mathbb{R}^{d+1}_+)$ and we have
\begin{equation}\label{ineqtransl}
\left\| \tau^\alpha_x\varphi\right\|_{\alpha,p}\leq \left\|\varphi\right\|_{\alpha,p}.
\end{equation}
\end{prop}
\begin{prop}
   Let $\varphi\in L^1_{\alpha}(\mathbb{R}^{d+1}_+)$. Then for all $x\in \mathbb{R}^{d+1}_+$,
    \begin{equation}\label{integraltransrad}
      \int_{\mathbb{R}^{d+1}_+}\tau^\alpha_x\varphi(y)d\mu_\alpha(y)= \int_{\mathbb{R}^{d+1}_+}\varphi(y)  d\mu_\alpha(y).
    \end{equation}
  \end{prop}
  \begin{proof}
    The result comes from combination identities (\ref{inv}) and (\ref{MMM}).
  \end{proof}
  By using the generalized translation, we define the generalized convolution product $\varphi*\psi$ of the functions $\varphi,\;\psi\in L^1_\alpha(\mathbb{R}^{d+1}_+)$ as follows
\begin{equation}\label{defconvolution}
\varphi*\psi(x)=\int_{\mathbb{R}^{d+1}_+}\tau^\alpha_x\varphi(-y)\psi(y)d\mu_\alpha(y).
\end{equation}
\\
This convolution is commutative and associative, and it satisfies the following properties.

\begin{prop}\label{propconvol}
i) For all $\varphi,\psi\in L^1_\alpha(\mathbb{R}^{d+1}_+),$\;(resp. $\varphi,\psi\in \mathcal{S}_*(\mathbb{R}^{d+1})$), then $\varphi*\psi\in L^1_\alpha(\mathbb{R}^{d+1}_+),$\;(resp. $\varphi*\psi\in \mathcal{S}_*(\mathbb{R}^{d+1})$) and we have
\begin{equation}\label{F(f*g)}
\mathcal{F}_{W}(\varphi*\psi)=\mathcal{F}_{W}(\varphi)\mathcal{F}_{W}(\psi).
\end{equation}
ii) Let $p, q, r\in [1,\infty],$ such that $\frac{1}{p}+\frac{1}{q}-\frac{1}{r}=1.$ Then for all $\varphi\in L^p_\alpha(\mathbb{R}^{d+1}_+)$ and  $\psi\in L^q_\alpha(\mathbb{R}^{d+1}_+)$ the function $\varphi*\psi$ belongs to  $L^r_\alpha(\mathbb{R}^{d+1}_+)$ and we have
\begin{equation}\label{invconvol2}
\left\|\varphi*\psi\right\|_{\alpha,r}\leq\left\|\varphi\right\|_{\alpha,p}\left\|\psi\right\|_{\alpha,q}.
\end{equation}
iii)  Let $\varphi,\psi\in L^2_\alpha(\mathbb{R}^{d+1}_+)$. Then
\begin{equation}
\varphi*\psi=\mathcal{F}_{W}^{-1}\left(\mathcal{F}_{W}(\varphi)\mathcal{F}_{W}(\psi)\right).
\end{equation}
iv) Let $\varphi,\psi\in L^2_\alpha(\mathbb{R}^{d+1}_+)$. Then $\varphi*\psi$ belongs to $L^2_\alpha(\mathbb{R}^{d+1}_+)$ if and only if $\mathcal{F}_{W}(\varphi)\mathcal{F}_{W}(\psi)$ belongs to $L^2_\alpha(\mathbb{R}^{d+1}_+)$ and we have
\begin{equation}
\mathcal{F}_{W}(\varphi*\psi)=\mathcal{F}_{W}(\varphi)\mathcal{F}_{W}(\psi).
\end{equation}
v)  Let $\varphi,\psi\in L^2_\alpha(\mathbb{R}^{d+1}_+)$. Then
\begin{equation}
\|\varphi*\psi\|_{\alpha,2}=\|\mathcal{F}_{W}(\varphi)\mathcal{F}_{W}(\psi)\|_{\alpha,2},
\end{equation}
where both sides are finite or infinite.
\end{prop}
\subsection{Weinstein two-wavelet theory}

In the following, we denote by \\
$\X=\left\{(a,x): x\in \mathbb{R}^{d+1}_+ \;\text{and}\; a>0\right\}$.\\
$L^p_{\alpha}(\X),\; p\in [1,\infty]$ the space of measurable functions $\varphi$ on $\X$ such that
\begin{eqnarray*}
  \|\varphi\|_{L^p_{\alpha}(\X)} &=& \left(\int_{\X}|\varphi(a,x)|^p  d\mu_\alpha(a,x)\right)^\frac{1}{p}<\infty,\quad 1\leq p<\infty, \\
  \|\varphi\|_{L^\infty_{\alpha}(\X)}&=& \esssup_{(a,x)\in\X}|\varphi(a,x)|<\infty,
\end{eqnarray*}
where the measure $\mu_\alpha(a,x)$ is defined on $\X$ by
$$d\mu_\alpha(a,x)=\frac{d\mu_\alpha(x)da}{a^{2\alpha+d+3}}.$$
\begin{defn}\cite{gasmi2016inversion}
 A classical wavelet  on $\mathbb{R}^{d+1}_+$ is a measurable function $\varphi$ on $\mathbb{R}^{d+1}_+$
satisfying for almost all $\xi\in \mathbb{R}^{d+1}_+$, the condition
\begin{equation}\label{defwave}
  0<C_\varphi=\int_{0}^{\infty}|\mathcal{F}_{W}(\varphi)(a\xi)|^2\frac{da}{a}<\infty.
\end{equation}
\end{defn}
We extend the notion of the wavelet to the  two-wavelet in Weinstein setting as follows.
\begin{defn}
  Let $\varphi$ and $\psi$ be in $L^2_{\alpha}(\mathbb{R}^{d+1}_+)$. We say that the pair $(\varphi,\psi)$ is a Weinstein
two-wavelet on $\mathbb{R}^{d+1}_+$ if the following integral
\begin{equation}\label{deftwowave}
  C_{\varphi,\psi}=\int_{0}^{\infty}\mathcal{F}_{W}(\psi)(a\xi)\overline{\mathcal{F}_{W}(\varphi)
  (a\xi)}\frac{da}{a}
\end{equation}
is constant for almost all  $\xi\in \mathbb{R}^{d+1}_+$ and we call the number $ C_{\varphi,\psi}$ the Weinstein two-wavelet constant associated to the functions $\varphi$ and $\psi$.
\end{defn}
It is to highlight that if $\varphi$ is a Weinstein  wavelet then the pair $(\varphi,\psi)$ is a Weinstein two-wavelet,
and $C_{\varphi,\psi}$ coincides with   $ C_{\varphi}$.

Let $a>0$ and $\varphi$ be a measurable function. We consider the function $\varphi_a$ defined by
\begin{equation}\label{fia}
  \forall x\in \mathbb{R}^{d+1}_+, \quad \varphi_a(x)=\frac{1}{a^{2\alpha+d+2}}\varphi\left(\frac{x}{a}\right).
\end{equation}
\begin{prop}
  \begin{enumerate}
    \item Let $a>0$ and $\varphi\in L^p_{\alpha}(\mathbb{R}^{d+1}_+),\;p\in[1,\infty]$. The function $\varphi_a$ belongs to $L^p_{\alpha}(\mathbb{R}^{d+1}_+)$ and we have
        \begin{equation}\label{normLpfia}
           \|\varphi_a\|_{\alpha,p}=a^{(2\alpha+d+2)(\frac{1}{p}-1)} \|\varphi\|_{\alpha,p}.
        \end{equation}
    \item Let $a>0$ and $\varphi\in L^1_{\alpha}(\mathbb{R}^{d+1}_+)\cup L^2_{\alpha}(\mathbb{R}^{d+1}_+)$. Then, we have
        \begin{equation}\label{fourierfia}
          \mathcal{F}_{W}(\varphi_a)(\xi)=\mathcal{F}_{W}(\varphi)(a\xi),\quad\xi\in \mathbb{R}^{d+1}_+.
        \end{equation}
  \end{enumerate}
\end{prop}
For $a>0$ and $\varphi\in L^2_{\alpha}(\mathbb{R}^{d+1}_+)$, we consider the family $\varphi_{a,x},\; x\in
\mathbb{R}^{d+1}_+$ of Weinstein wavelets on $\mathbb{R}^{d+1}_+$ in $L^2_{\alpha}(\mathbb{R}^{d+1}_+)$ defined by
\begin{equation}\label{deffiax}
  \forall y\in \mathbb{R}^{d+1}_+,\quad \varphi_{a,x}=a^{\alpha+1+\frac{d}{2}}\tau^\alpha_x\varphi_a(y).
\end{equation}
\begin{rem}
  \begin{enumerate}
    \item   Let $\varphi$ be a function in $L^2_{\alpha}(\mathbb{R}^{d+1}_+)$, then we have
    \begin{equation}\label{Nfiax2}
      \forall (a,x)\in\X, \quad \|\varphi_{a,x}\|_{\alpha,2}\leq  \|\varphi\|_{\alpha,2}.
    \end{equation}
    \item  Let $p\in [1,\infty]$ and $\varphi$ be a function in $L^p_{\alpha}(\mathbb{R}^{d+1}_+)$, then we have
    \begin{equation}\label{Nfiaxp}
      \forall (a,x)\in\X, \quad \|\varphi_{a,x}\|_{\alpha,p}\leq a^{(2\alpha+d+2)(\frac{1}{p}-\frac{1}{2})} \|\varphi\|_{\alpha,p}.
    \end{equation}
  \end{enumerate}
\end{rem}
\begin{defn}\cite{mejjaoli2017new}
  Let $\varphi$ be a Weinstein wavelet on $\mathbb{R}^{d+1}_+$ in $L^2_{\alpha}(\mathbb{R}^{d+1}_+)$. The Weinstein continuous wavelet transform $\Phi^W_\varphi$ on $\mathbb{R}^{d+1}_+$ is defined for regular functions $f$ on $\mathbb{R}^{d+1}_+$ by
  \begin{equation}\label{contwave}
    \forall (a,x)\in\X,\quad \Phi^W_\varphi(f)(a,x)=\int_{\mathbb{R}^{d+1}_+}f(y)\overline{\varphi_{a,x}(y)} d\mu_\alpha(y)=\langle f, \varphi_{a,x}\rangle_{\alpha,2}.
  \end{equation}
\end{defn}
This transform can also be written in the form
\begin{equation}\label{recontwave}
  \Phi^W_\varphi(f)(a,x)=a^{\alpha+1+\frac{d}{2}}\check{f}*\overline{\varphi_a}(x).
\end{equation}
\begin{rem}
  \begin{enumerate}
    \item Let $\varphi$ be a function in $L^p_{\alpha}(\mathbb{R}^{d+1}_+)$, and Let $f$ be a function in $L^q_{\alpha}(\mathbb{R}^{d+1}_+)$, with $p\in [1,\infty]$, we define the Weinstein continuous wavelet transform $\Phi^W_\varphi(f)$ by the relation (\ref{recontwave}).
    \item Let $\varphi$ be a Weinstein wavelet on $\mathbb{R}^{d+1}_+$ in $L^2_{\alpha}(\mathbb{R}^{d+1}_+)$. Then from the relations (\ref{Nfiax2}) and (\ref{contwave}), we have for all $f\in L^2_{\alpha}(\mathbb{R}^{d+1}_+)$
        \begin{equation}\label{Nwcont2}
           \|\Phi^W_\varphi(f)\|_{\alpha,\infty}\leq \|f\|_{\alpha,2}\|\varphi\|_{\alpha,2}.
        \end{equation}
    \item Let $\varphi$ be a function in $L^p_{\alpha}(\mathbb{R}^{d+1}_+)$, with $p\in [1,\infty]$, then from the inequality (\ref{invconvol2}) and the identity (\ref{recontwave}),  we have for all $f\in L^q_{\alpha}(\mathbb{R}^{d+1}_+)$
        \begin{equation}\label{Nwcontp}
           \|\Phi^W_\varphi(f)\|_{\alpha,\infty}\leq \|f\|_{\alpha,q}\|\varphi\|_{\alpha,p}.
        \end{equation}
  \end{enumerate}
\end{rem}
\begin{thm}(Parseval's formula)\label{Parseval's formula2wave}\cite{saoudi2020two}
  Let $(\varphi,\psi)$ be a Weinstein two-wavelet. Then for all $\varphi$ and $\psi$ in $L^2_{\alpha}(\mathbb{R}^{d+1}_+)$, we have the following  Parseval type formula
  \begin{equation}\label{ParsevalPsi}
    \int_{\X}\Phi^W_\varphi(f)(a,x)\overline{\Phi^W_\psi(g)(a,x)}d\mu_\alpha(a,x)=C_{\varphi,\psi}\int_{\mathbb{R}^{d+1}_+}f(x)\overline{g(x)} d\mu_\alpha(x),
  \end{equation}
  where $ C_{\varphi,\psi}$ is the Weinstein two-wavelet constant associated to the functions $\varphi$ and $\psi$ given by the identity (\ref{deftwowave}).
\end{thm}
\begin{cor}\cite{saoudi2020two}
  Let $(\varphi,\psi)$ be a Weinstein two-wavelet. Then we have the following assertion:
 If the Weinstein two-wavelet constant  $C_{\varphi,\varphi}=0$, then $\Phi^W_\varphi\left(L^2_{\alpha}(\mathbb{R}^{d+1}_+)\right)$ and $\Phi^W_\psi\left(L^2_{\alpha}(\mathbb{R}^{d+1}_+)\right)$ are orthogonal.
\end{cor}
\begin{thm}(Inversion formula)\cite{saoudi2020two} Let $(\varphi,\psi)$ be a Weinstein two-wavelet. For all $f\in L^1_{\alpha}(\mathbb{R}^{d+1}_+)$ (resp. $L^2_{\alpha}(\mathbb{R}^{d+1}_+)$) such that $\mathcal{F}_{W}(f)$  belongs to $f\in L^1_{\alpha}(\mathbb{R}^{d+1}_+)$
  (resp. $L^1_{\alpha}(\mathbb{R}^{d+1}_+)$ $\cap L^\infty_{\alpha}(\mathbb{R}^{d+1}_+)$), we have
  \begin{equation}\label{inversion2wave}
    f(y)=\frac{1}{C_{\varphi,\psi}}\int_{0}^{\infty}\int_{\mathbb{R}^{d+1}_+}\Phi^W_\varphi(f)(a,x)\psi_{a,x}(y)d\mu_\alpha(a,x),
  \end{equation}
where for each $y\in\mathbb{R}^{d+1}_+ $\, both the inner integral and the outer integral are absolutely
convergent, but eventually not the double integral.
\end{thm}
\section{The Weinstein two-wavelet localization operators}
In this section, we will give a host of sufficient conditions for the boundedness and compactness of the two-wavelet localization operator $\mathcal{L}_{\varphi,\psi}(\sigma)$ on $L^{p}_{\alpha}(\mathbb{R}^{d+1}_+)$ for all $1\leq p\leq \infty$, in terms of properties of the symbol $\sigma$ and the functions $\varphi$ and $\psi$.
\begin{defn} Let $\varphi, \psi$ be measurable functions on $\mathbb{R}_{+}^{d+1}$, $\sigma$ be measurable function on $\mathcal{X}$, we the two-wavelet localization operator noted by $\mathcal{L}_{\varphi,\psi}(\sigma )$, on $L^p_\alpha(\mathbb{R}^{d+1}_+)$, $1 \leq p\leq \infty$, by
\begin{equation}\label{two-waveloc}
\mathcal{L}_{\varphi,\psi}(\sigma)(f)(y)=\int_{\X}
 \sigma(a,x)\Phi^W_\varphi(f)(a,x)\psi_{a,x}(y)d\mu_\alpha(a,x),\ \ y\in\mathbb{R}_{+}^{d+1}.
\end{equation}
\end{defn}
In accordance with the different choices of the symbols $\sigma$ and the different continuities
required, we need to impose different conditions on the functions $\varphi$ and $\psi$, and then we obtain a
two-wavelet localization operator on $L^p_\alpha(\mathbb{R}^{d+1}_+)$.
It is often more practical to interpret the definition of the localization operator in a weak sense as following: for all $f$ in $L^p_\alpha(\mathbb{R}^{d+1}_+)$, $1 \leq p\leq \infty$, and $f$ in $L^q_\alpha(\mathbb{R}^{d+1}_+)$
\begin{equation}\label{weaktwo-waveloc}
\langle\mathcal{L}_{\varphi,\psi}(\sigma)(f),g\rangle_{\alpha,2}=\int_{\X}
 \sigma(a,x)\Phi^W_\varphi(f)(a,x)\overline{\Phi^W_\psi(g)(a,x)}d\mu_\alpha(a,x),\ \ y\in\mathbb{R}_{+}^{d+1}.
\end{equation}

\begin{prop}   Let $1 \leq p\leq \infty$. Then the adjoint of the two-wavelet localization operator
$$\mathcal{L}_{\varphi, \psi}(\sigma):\ L^p_\alpha(\mathbb{R}^{d+1}_+)\ \longrightarrow\ L^p_\alpha(\mathbb{R}^{d+1}_+),$$ is $\mathcal{L}_{\psi, \varphi}(\overline{\sigma}):\ L^q_\alpha(\mathbb{R}^{d+1}_+)\ \longrightarrow\ L^q_\alpha(\mathbb{R}^{d+1}_+)$.
\end{prop}
\begin{proof} Let $f$ in $L^p_\alpha(\mathbb{R}^{d+1}_+)$ and $g$ in $L^q_\alpha(\mathbb{R}^{d+1}_+)$. Then we have from the relation (\ref{weaktwo-waveloc})
\begin{eqnarray*}
\langle \mathcal{L}_{\varphi, \psi}(\sigma)(f),\ g \rangle_{\alpha,2}&=& \int_{\X}
 \sigma(a,x)\Phi^W_\varphi(f)(a,x)\overline{\Phi^W_\psi(g)(a,x)}d\mu_\alpha(a,x)\\
&=&\overline{\int_{\X}
 \overline{\sigma(a,x)}\Phi^W_\psi(f)(a,x)\overline{\Phi^W_\varphi(g)(a,x)}d\mu_\alpha(a,x)}
\\&=&\overline{\langle \mathcal{L}_{\psi, \varphi}(\overline{\sigma})(g),\ f \rangle_{\alpha,2}}
\\&=&\langle f, \mathcal{L}_{\psi, \varphi}(\overline{\sigma})(g) \rangle_{\alpha,2}.
\end{eqnarray*}
Thus,
\begin{eqnarray}\label{lphp}
\mathcal{L}_{\varphi, \psi}^{*}(\sigma)=\mathcal{L}_{\psi, \varphi}(\overline{\sigma}).
\end{eqnarray}
\end{proof}
\subsection{$L_\alpha^p$-Boundedness of $\mathcal{L}_{\varphi,\psi}(\sigma)$}
For $1 \leq p\leq \infty$, put $\sigma\in  L^1_\alpha(\X)$, $\psi \in L^p_\alpha(\mathbb{R}^{d+1}_+)$ and $\varphi\in L^q_\alpha(\mathbb{R}^{d+1}_+)$. In this subsection, we  are going to show that the Weinstein two-wavelet localization operator
$\mathcal{L}_{\varphi,\psi}(\sigma)$ is a bounded operator on $L^p_\alpha(\mathbb{R}^{d+1}_+)$.
\begin{prop}\label{prop31}
  Let $\sigma\in L^{1}_{\alpha}(\X)$, $\varphi\in L^{\infty}_{\alpha}(\mathbb{R}^{d+1}_+)$
  and  $\psi\in L^{1}_{\alpha}(\mathbb{R}^{d+1}_+)$, then the localization
operator $\mathcal{L}_{\varphi,\psi}(\sigma)$  is bounded and linear from $L^{1}_{\alpha}(\mathbb{R}^{d+1}_+)$ onto itself and we have
\begin{equation*}
  \|\mathcal{L}_{\varphi,\psi}(\sigma)\|_{\mathcal{B}(L^{1}_{\alpha}(\mathbb{R}^{d+1}_+))}\leq \|\varphi\|_{\alpha,\infty} \|\psi\|_{\alpha,1}
  \|\sigma\|_{L^{1}_{\alpha}(\X)}.
\end{equation*}
\end{prop}
\begin{proof}
  Let $f$ be a function in $L^{1}_{\alpha}(\X)$. From the definition of the two-wavelet localization operator (\ref{two-waveloc})
 and according to relations (\ref{deffiax}), (\ref{recontwave}),  (\ref{invconvol2}) and (\ref{ineqtransl}), we have
\begin{eqnarray*}
  \|\mathcal{L}_{\varphi,\psi}(\sigma)(f)\|_{\alpha,1} &\leq & \int_{\mathbb{R}^{d+1}_+} \int_{\X}
 |\sigma(a,x)| \, |\Phi^W_\varphi(f)(a,x)| \, |\psi_{a,x}(y)|d\mu_\alpha(a,x)d\mu_\alpha(y)\\
   &\leq &  \int_{\mathbb{R}^{d+1}_+} \int_{\X}
 |\sigma(a,x)| \, |\check{f}*\overline{\varphi_a(x)}| \, |\tau^\alpha_x\psi_{a}(y)|d\mu_\alpha(x)\frac{da}{a} d\mu_\alpha(y) \\
   &\leq & \int_{\mathbb{R}^{d+1}_+} \int_{\X}
 |\sigma(a,x)| \, \|f\|_{\alpha,1}\|\varphi_a\|_{\alpha,\infty} \, |\tau^\alpha_x\psi_{a}(y)|d\mu_\alpha(x)\frac{da}{a} d\mu_\alpha(y) \\
   &\leq & \|f\|_{\alpha,1}\|\varphi\|_{\alpha,\infty} \int_{\X}|\sigma(a,x| \left[\int_{\mathbb{R}^{d+1}_+} |\tau^\alpha_x\psi_{a}(y)|d\mu_\alpha(y) \right] d\mu_\alpha(x)\frac{da}{a^{2\alpha+d+3}} \\
   &\leq &  \|f\|_{\alpha,1}\|\varphi\|_{\alpha,\infty} \int_{\X}|\sigma(a,x| \left[\int_{\mathbb{R}^{d+1}_+} |\psi_{a}(y)|d\mu_\alpha(y) \right] d\mu_\alpha(x)\frac{da}{a^{2\alpha+d+3}} \\
   &\leq & \|f\|_{\alpha,1}\|\varphi\|_{\alpha,\infty}\|\psi\|_{\alpha,1}\|\sigma\|_{L^{1}_{\alpha}(\X)}.
\end{eqnarray*}
Therefore,
\begin{equation*}
  \|\mathcal{L}_{\varphi,\psi}(\sigma)\|_{\mathcal{B}(L^{1}_{\alpha}(\mathbb{R}^{d+1}_+))}\leq \|\varphi\|_{\alpha,\infty} \|\psi\|_{\alpha,1}
  \|\sigma\|_{L^{1}_{\alpha}(\X)}.
\end{equation*}
\end{proof}
\begin{prop}\label{prop32}
  Let $\sigma\in L^{1}_{\alpha}(\X)$, $\varphi\in L^{1}_{\alpha}(\mathbb{R}^{d+1}_+)$
  and  $\psi\in L^{\infty}_{\alpha}(\mathbb{R}^{d+1}_+)$, then the localization
operator $\mathcal{L}_{\varphi,\psi}(\sigma)$  is bounded and linear from $L^{\infty}_{\alpha}(\mathbb{R}^{d+1}_+)$ onto itself and we have
\begin{equation*}
  \|\mathcal{L}_{\varphi,\psi}(\sigma)\|_{\mathcal{B}(L^{\infty}_{\alpha}(\mathbb{R}^{d+1}_+))}\leq \|\varphi\|_{\alpha,1} \|\psi\|_{\alpha,\infty}
  \|\sigma\|_{L^{1}_{\alpha}(\X)}.
\end{equation*}
\end{prop}
\begin{proof}
  Let $f$ be a function in $L^{\infty}_{\alpha}(\X)$. From the definition of the two-wavelet localization operator (\ref{two-waveloc})
 and as above, according to relations (\ref{deffiax}), (\ref{recontwave}),  (\ref{invconvol2}) and (\ref{ineqtransl}), we have for all $y\in \mathbb{R}^{d+1}_+$
\begin{eqnarray*}
  |\mathcal{L}_{\varphi,\psi}(\sigma)(f)(y)|&\leq & \int_{\X}
 |\sigma(a,x)|\,|\Phi^W_\varphi(f)(a,x)|\,|\psi_{a,x}(y)|d\mu_\alpha(a,x) \\
   &\leq & \|f\|_{\alpha,\infty} \|\varphi\|_{\alpha,1} \|\psi\|_{\alpha,\infty} \|\sigma\|_{L^{1}_{\alpha}(\X)}.
\end{eqnarray*}
Thus,
\begin{equation*}
  \|\mathcal{L}_{\varphi,\psi}(\sigma)\|_{\mathcal{B}(L^{\infty}_{\alpha}(\mathbb{R}^{d+1}_+))}\leq \|\varphi\|_{\alpha,\infty} \|\psi\|_{\infty,1}
  \|\sigma\|_{L^{1}_{\alpha}(\X)}.
\end{equation*}
\end{proof}
By interpolations of the results of Propositions \ref{prop31} and \ref{prop32}, we get the following result.
\begin{thm}\label{thm3.3}
  Let $\varphi$ and $\psi$ be functions in  $L^{1}_{\alpha}(\mathbb{R}^{d+1}_+)\cap L^{\infty}_{\alpha}(\mathbb{R}^{d+1}_+)$. Then for all $\sigma\in L^{1}_{\alpha}(\X)$, there exists a unique bounded linear operator $\mathcal{L}_{\varphi,\psi}(\sigma)$ from $L^{p}_{\alpha}(\mathbb{R}^{d+1}_+)$ onto itself with $1\leq p\leq \infty$, such that

\begin{equation*}
  \|\mathcal{L}_{\varphi,\psi}(\sigma)\|_{\mathcal{B}(L^{p}_{\alpha}(\mathbb{R}^{d+1}_+))}\leq \|\varphi\|_{\alpha,1}^\frac{1}{q} \|\psi\|_{\alpha,1}^\frac{1}{p} \|\varphi\|_{\alpha,\infty}^\frac{1}{p} \|\psi\|_{\alpha,\infty}^\frac{1}{q}
  \|\sigma\|_{L^{1}_{\alpha}(\X)}.
\end{equation*}
  \end{thm}
In the following proposition, we generalize and we improve Proposition \ref{prop32}.
\begin{prop}\label{proppq}
  Let $\sigma\in L^{1}_{\alpha}(\X)$, $\psi\in L^{p}_{\alpha}(\mathbb{R}^{d+1}_+)$
  and  $\varphi\in L^{q}_{\alpha}(\mathbb{R}^{d+1}_+)$, for $1< p \leq \infty$ then the localization
operator $\mathcal{L}_{\varphi,\psi}(\sigma)$  is bounded linear operator from $L^{p}_{\alpha}(\mathbb{R}^{d+1}_+)$ onto itself and we have
\begin{equation*}
  \|\mathcal{L}_{\varphi,\psi}(\sigma)\|_{\mathcal{B}(L^{p}_{\alpha}(\mathbb{R}^{d+1}_+))}\leq \|\varphi\|_{\alpha,q} \|\psi\|_{\alpha,p}
  \|\sigma\|_{L^{1}_{\alpha}(\X)}.
\end{equation*}
\end{prop}
\begin{proof}
  Let $f$ be a function in $L^{p}_{\alpha}(\mathbb{R}^{d+1}_+)$. We consider the linear functional
  \begin{eqnarray*}
  \begin{array}{lccl}
    \mathcal{J}_f:& L^{q}_{\alpha}(\mathbb{R}^{d+1}_+)  &\longrightarrow& \mathbb{C} \\
    &g &\longmapsto& \langle g,\mathcal{L}_{\varphi,\psi}(\sigma)(f)\rangle_{\alpha,2}.
  \end{array}
\end{eqnarray*}
According to relation (\ref{weaktwo-waveloc}), we have
\begin{eqnarray*}
 |\langle \mathcal{L}_{\varphi, \psi}(\sigma)(f), g \rangle_{\alpha,2}| &\leq& \int_{\X}
 |\sigma(a,x)|\,|\Phi^W_\varphi(f)(a,x)|\,|\overline{\Phi^W_\psi(g)(a,x)}|d\mu_\alpha(a,x) \\
  &\leq &  \|\Phi^W_\varphi(f)\|_{L^{\infty}_{\alpha}(\X)} \|\Phi^W_\psi(g)\|_{L^{\infty}_{\alpha}(\X)} \|\sigma\|_{L^{1}_{\alpha}(\X)}.
\end{eqnarray*}
Next, using the relations (\ref{recontwave}) and  (\ref{invconvol2}), we get
$$ |\langle \mathcal{L}_{\varphi, \psi}(\sigma)(f), g \rangle_{\alpha,2}|\leq  \|\sigma\|_{L^{1}_{\alpha}(\X)}\|\varphi\|_{\alpha,q} \|\psi\|_{\alpha,p} \|f\|_{\alpha,p} \|g\|_{\alpha,q}.$$
Therefore, the operator $\mathcal{J}_f$ is a continuous linear functional on $L^{q}_{\alpha}(\mathbb{R}^{d+1}_+)$, and we have
$$\|\mathcal{J}_f\|_{\mathcal{B}(L^{q}_{\alpha}(\mathbb{R}^{d+1}_+))}\leq \|\sigma\|_{L^{1}_{\alpha}(\X)}\|\varphi\|_{\alpha,q} \|\psi\|_{\alpha,p} \|f\|_{\alpha,p} .$$
Like that $\mathcal{J}_f(g)=\langle g,\mathcal{L}_{\varphi,\psi}(\sigma)(f)\rangle_{\alpha,2}$, then by the Riesz representation theorem, we have
$$  \|\mathcal{L}_{\varphi,\psi}(\sigma)(f)\|_{\alpha,p}= \|\mathcal{J}_f\|_{\mathcal{B}(L^{q}_{\alpha}(\mathbb{R}^{d+1}_+))}\leq \|\sigma\|_{L^{1}_{\alpha}(\X)}\|\varphi\|_{\alpha,q} \|\psi\|_{\alpha,p} \|f\|_{\alpha,p},$$
which completes the proof.
\end{proof}
By combining the results obtained in Propositions \ref{prop31} and \ref{proppq}, we have the following version of the $L^{p}_{\alpha}$-boundedness result.
\begin{thm}\label{thm6}
   Let $\sigma\in L^{1}_{\alpha}(\X)$, $\psi\in L^{p}_{\alpha}(\mathbb{R}^{d+1}_+)$
  and  $\varphi\in L^{q}_{\alpha}(\mathbb{R}^{d+1}_+)$, for $1\leq p \leq \infty$, then the localization
operator $\mathcal{L}_{\varphi,\psi}(\sigma)$  is a bounded and linear from $L^{p}_{\alpha}(\mathbb{R}^{d+1}_+)$ onto itself, and we have
\begin{equation*}
  \|\mathcal{L}_{\varphi,\psi}(\sigma)\|_{\mathcal{B}(L^{p}_{\alpha}(\mathbb{R}^{d+1}_+))}\leq \|\varphi\|_{\alpha,q} \|\psi\|_{\alpha,p}
  \|\sigma\|_{L^{1}_{\alpha}(\X)}.
\end{equation*}
\end{thm}
According to Schur technique, we can obtain an $L_\alpha^p$-boundedness result as in the previous
Theorem with crude estimate of the norm $ \|\mathcal{L}_{\varphi,\psi}(\sigma)\|_{\mathcal{B}(L^{p}_{\alpha}(\mathbb{R}^{d+1}_+))}$.

\begin{thm}\label{thmR}
  Let $\sigma\in L^{1}_{\alpha}(\X)$, $\varphi$ and $\psi$ in  $L^{1}_{\alpha}(\mathbb{R}^{d+1}_+)\cap L^{\infty}_{\alpha}(\mathbb{R}^{d+1}_+)$. Then there exists a unique bounded linear operator  $\mathcal{L}_{\varphi,\psi}(\sigma)$ from $L^{p}_{\alpha}(\mathbb{R}^{d+1}_+)$ onto itself with $1\leq p\leq \infty$ and we have
  \begin{equation*}
  \|\mathcal{L}_{\varphi,\psi}(\sigma)\|_{\mathcal{B}(L^{p}_{\alpha}(\mathbb{R}^{d+1}_+))}\leq \max\left( \|\varphi\|_{\alpha,1} \|\psi\|_{\alpha,\infty}, \|\varphi\|_{\alpha,\infty} \|\psi\|_{\alpha,1}\right)
  \|\sigma\|_{L^{1}_{\alpha}(\X)}.
\end{equation*}
  \end{thm}
\begin{proof}
We put the function $\mathcal{R}$ defined on $\mathbb{R}^{d+1}_+\times\mathbb{R}^{d+1}_+$ by
$$\mathcal{R}(y,z)=\int_{\X}
\sigma(a,x)\overline{\varphi_{a,x}(z)} \psi_{a,x}(y) d\mu_{\alpha}(a,x).$$
Then the  Weinstein two-wavelet localization operator  can be written
in terms of $\mathcal{R}(y,z)$ as follows
$$\mathcal{L}_{\varphi,\psi}(\sigma)(f)(y)=\int_{\mathbb{R}_{+}^{d+1}}\mathcal{R}(y,z)f(z) d\mu_{\alpha}(z).$$
Next, it is easy to see that
$$\int_{\mathbb{R}_{+}^{d+1}}|\mathcal{R}(y,z)| d\mu_{\alpha}(y)\leq \|\varphi\|_{\alpha,\infty} \|\psi\|_{\alpha,1}\|\sigma\|_{L^{1}_{\alpha}(\X)},\quad z\in \mathbb{R}_{+}^{d+1}, $$
and
$$\int_{\mathbb{R}_{+}^{d+1}}|\mathcal{R}(y,z)| d\mu_{\alpha}(z)\leq \|\varphi\|_{\alpha,1} \|\psi\|_{\alpha,\infty} \|\sigma\|_{L^{1}_{\alpha}(\X)}, \quad y\in \mathbb{R}_{+}^{d+1}.$$
Then by Schur Lemma (cf.\cite{folland1995introduction}), we can conclude that the localization operator $\mathcal{L}_{\varphi,\psi}(\sigma)$ is bounded and linear from $L^{p}_{\alpha}(\mathbb{R}^{d+1}_+)$ onto itself for all $1\leq p\leq \infty$, and we have
\begin{equation*}
  \|\mathcal{L}_{\varphi,\psi}(\sigma)\|_{\mathcal{B}(L^{p}_{\alpha}(\mathbb{R}^{d+1}_+))}\leq \max\left( \|\varphi\|_{\alpha,1} \|\psi\|_{\alpha,\infty}, \|\varphi\|_{\alpha,\infty} \|\psi\|_{\alpha,1}\right)
  \|\sigma\|_{L^{1}_{\alpha}(\X)}.
\end{equation*}
\end{proof}
The previous Theorem tells us that the unique bounded linear operator on the spaces $L^{p}_{\alpha}(\mathbb{R}^{d+1}_+),$ $1\leq p\leq \infty$, obtained in Theorem \ref{thmR},  is in fact the
integral operator on $L^{p}_{\alpha}(\mathbb{R}^{d+1}_+),$ $1\leq p\leq \infty$ with kernel $\mathcal{R}$.

Subsequently, we can now state and prove the main result in this subsection.

\begin{thm}\label{thm3.7}
 Let $\sigma\in L^{r}_{\alpha}(\X)$, $r\in [1,2]$, and  $\varphi,\psi$ in  $L^{1}_{\alpha}(\mathbb{R}^{d+1}_+)\cap L^{2}_{\alpha}(\mathbb{R}^{d+1}_+)\cap L^{\infty}_{\alpha}(\mathbb{R}^{d+1}_+)$. Then there exists a unique bounded linear operator  $\mathcal{L}_{\varphi,\psi}(\sigma)$ from $L^{p}_{\alpha}(\mathbb{R}^{d+1}_+)$ onto itself for all $p\in [r,r']$ and we have
  \begin{equation*}
  \|\mathcal{L}_{\varphi,\psi}(\sigma)\|_{\mathcal{B}(L^{p}_{\alpha}(\mathbb{R}^{d+1}_+))}\leq K_1^tK_2^{1-t}
  \|\sigma\|_{L^{r}_{\alpha}(\X)},
\end{equation*}
where
\begin{eqnarray*}
  K_1 &=& \left(\|\varphi\|_{\alpha,\infty} \|\psi\|_{\alpha,1}\right)^{\frac{2}{r}-1} \left(\sqrt{C_\varphi C_\psi}\|\varphi\|_{\alpha,2} \|\psi\|_{\alpha,2}\right)^\frac{1}{r'}, \\
  K_2 &=& \left(\|\varphi\|_{\alpha,1} \|\psi\|_{\alpha,\infty}\right)^{\frac{2}{r}-1} \left(\sqrt{C_\varphi C_\psi}\|\varphi\|_{\alpha,2} \|\psi\|_{\alpha,2}\right)^\frac{1}{r'},
\end{eqnarray*}
and $$\frac{t}{r}+\frac{1-t}{r'}=\frac{1}{p}.$$
\end{thm}
\begin{proof}
  Consider the linear functional
  \begin{eqnarray*}
  \begin{array}{lccl}
    \mathcal{J}: (L^{1}_{\alpha}(\X)\cap L^{2}_{\alpha}(\X)\times (L^{1}_{\alpha}(\mathbb{R}^{d+1}_+)\,\cap & \!\!\! L^{2}_{\alpha}(\mathbb{R}^{d+1}_+))& \longrightarrow & L^{1}_{\alpha}(\mathbb{R}^{d+1}_+)\cap L^{2}_{\alpha}(\mathbb{R}^{d+1}_+) \\
& (\sigma,f) &\longmapsto & \mathcal{L}_{\varphi,\psi}(\sigma)(f). \end{array}
  \end{eqnarray*}
  According to Proposition \ref{prop31}, we obtain
  \begin{equation*}
    \|\mathcal{J}(\sigma,f)\|_{\alpha,1}\leq \|\varphi\|_{\alpha,\infty} \|\psi\|_{\alpha,1} \|f\|_{\alpha,1} \|\sigma\|_{L^{1}_{\alpha}(\X)}
  \end{equation*}
  and from \cite[Theorem 3.1]{mejjaoli2017new}, we have
   \begin{equation*}
    \|\mathcal{J}(\sigma,f)\|_{\alpha,2}\leq \left(\sqrt{C_{\varphi}C_{\psi}} \|\varphi\|_{\alpha,2} \|\psi\|_{\alpha,2}\right)^\frac{1}{2} \|f\|_{\alpha,2} \|\sigma\|_{L^{2}_{\alpha}(\X)}.
  \end{equation*}
Therefore, by  the multi-linear interpolation theory \cite[ Section 10.1]{calderon1964intermediate}, we obtain a unique bounded linear operator
 $$ \mathcal{J}: L^{r}_{\alpha}(\X)\times L^{r}_{\alpha}(\mathbb{R}^{d+1}_+) \longrightarrow L^{r}_{\alpha}(\mathbb{R}^{d+1}_+) $$
such that
\begin{equation}\label{Jr}
  \|\mathcal{J}(\sigma,f)\|_{\alpha,r}\leq K_1\|f\|_{\alpha,r} \|\sigma\|_{L^{r}_{\alpha}(\X)},
\end{equation}
where
$$K_1= \left(\|\varphi\|_{\alpha,\infty} \|\psi\|_{\alpha,1}\right)^\theta \left(\sqrt{C_{\varphi}C_{\psi}}\|\varphi\|_{\alpha,2} \|\psi\|_{\alpha,2}\right)^\frac{1-\theta}{2}$$
and
$$\frac{\theta}{1}+\frac{1-\theta}{2}=\frac{1}{r}.$$
By the definition of the linear functional $\mathcal{J}$, we have
  \begin{equation}\label{Lbr}
  \|\mathcal{L}_{\varphi,\psi}(\sigma)\|_{\mathcal{B}(L^{r}_{\alpha}(\mathbb{R}^{d+1}_+))}\leq \left(\|\varphi\|_{\alpha,\infty} \|\psi\|_{\alpha,1}\right)^{\frac{2}{r}-1} \left(\sqrt{C_{\varphi}C_{\psi}}\|\varphi\|_{\alpha,2} \|\psi\|_{\alpha,2}\right)^\frac{1}{r'}
  \|\sigma\|_{L^{r}_{\alpha}(\X)}.
\end{equation}
Like the adjoint of $\mathcal{L}_{\varphi,\psi}(\sigma)$ is $\mathcal{L}_{\psi,\varphi}(\overline{\sigma})$, therefore $\mathcal{L}_{\varphi,\psi}(\sigma)$ is a bounded linear map on $L^{r'}_{\alpha}(\mathbb{R}^{d+1}_+)$ with
its operator norm
 \begin{equation}\label{Jr'}
  \|\mathcal{L}_{\varphi,\psi}(\sigma)\|_{\mathcal{B}(L^{r'}_{\alpha}(\mathbb{R}^{d+1}_+))}= \|\mathcal{L}_{\psi,\varphi}(\overline{\sigma})\|_{\mathcal{B}(L^{r}_{\alpha}(\mathbb{R}^{d+1}_+))} \leq K_2 \|\sigma\|_{L^{r}_{\alpha}(\X)},
\end{equation}
where
$$ K_2 = \left(\|\varphi\|_{\alpha,1} \|\psi\|_{\alpha,\infty}\right)^{\frac{2}{r}-1} \left(\sqrt{C_{\varphi}C_{\psi}}\|\varphi\|_{\alpha,2} \|\psi\|_{\alpha,2}\right)^\frac{1}{r'}.$$
Finally, by interpolation of (\ref{Lbr}) and (\ref{Jr'}), we obtain that, for all $p\in [r,r']$,
\begin{equation*}
  \|\mathcal{L}_{\varphi,\psi}(\sigma)\|_{\mathcal{B}(L^{p}_{\alpha}(\mathbb{R}^{d+1}_+))}\leq K_1^tK_2^{1-t}
  \|\sigma\|_{L^{r}_{\alpha}(\X)},
\end{equation*}
with
$$\frac{t}{r}+\frac{1-t}{r'}=\frac{1}{p}.$$
\end{proof}

\subsection{$L^p_\alpha$-Compactness of $\mathcal{L}_{\varphi,\psi}(\sigma)$}

In this subsection, we establish the compactness of the Weinstein two-wavelet localization operators $\mathcal{L}_{\varphi,\psi}(\sigma)$ on $L^{p}_{\alpha}(\mathbb{R}^{d+1}_+)$, $1\leq p\leq \infty$. Let us start with the following proposition.

\begin{prop}\label{prop3.15}
  Under the same assumptions of Theorem \ref{thm3.3}, the Weinstein two-wavelet localization operator
  $$\mathcal{L}_{\varphi,\psi}(\sigma):L^{1}_{\alpha}(\mathbb{R}^{d+1}_+)\longrightarrow L^{1}_{\alpha}(\mathbb{R}^{d+1}_+)$$
  is compact.
\end{prop}
\begin{proof}
  Let $(f_n)_{n\in\mathbb{N}}\in L^{1}_{\alpha}(\mathbb{R}^{d+1}_+)$ a sequence of functions that converge weakly to $0$ in $L^{1}_{\alpha}(\mathbb{R}^{d+1}_+)$ as $n$ converge to $\infty$. So, to show  the compactness of of localization operators, it is enough to prove that
  \begin{equation*}
    \lim_{n\to \infty} \|\mathcal{L}_{\varphi,\psi}(\sigma)(f_n)\|_{\alpha,1}=0.
  \end{equation*}
  We have
\begin{eqnarray}\label{3.9}
  \|\mathcal{L}_{\varphi,\psi}(\sigma)(f_n)\|_{\alpha,1} \leq \int_{\mathbb{R}_{+}^{d+1}}  \int_{\X}
 |\sigma(a,x)| |\langle f_n,\varphi_{a,x}
 \rangle_{\alpha,2}||\psi_{a,x}(y)| d\mu_{\alpha}(a,x) d\mu_{\alpha}(y).
  \end{eqnarray}
  Using the fact that $(f_n)_{n\in\mathbb{N}}$ converge weakly to $0$ in $L^{1}_{\alpha}(\mathbb{R}^{d+1}_+)$ as $n$ converge to $\infty$, we deduce
  \begin{equation}\label{3.10}
   \forall a>0, \forall x,y,\in \mathbb{R}_{+}^{d+1},\quad\lim_{n\to \infty} |\sigma(a,x)| |\langle f_n,\varphi_{a,x}
 \rangle_{\alpha,2}||\psi_{a,x}(y)|=0.
  \end{equation}
  Moreover, as $f_n\rightharpoonup 0$  weakly  in $L^{1}_{\alpha}(\mathbb{R}^{d+1}_+)$,  then there exists a positive
constant $C$ such that $\|f_n\|_{\alpha,1}\leq C$. Hence,
$$\forall a>0, \forall x,y,\in \mathbb{R}_{+}^{d+1},\quad|\sigma(a,x)| |\langle f_n,\varphi_{a,x}
 \rangle_{\alpha,2}||\psi_{a,x}(y)| $$
 \begin{equation}\label{3.11}
   \leq C |\sigma(a,x)| \|\varphi\|_{\alpha,\infty}|\tau_x^\alpha\psi(y)|.
  \end{equation}
  On the other hand, from Fubini's theorem and relation (\ref{ineqtransl}), we obtain
   \begin{eqnarray}\label{3.12}
     \int_{\mathbb{R}_{+}^{d+1}} && \!\!\!\!\!\!\!\!\!\! \!\!\!\!\!\int_{\X}
 |\sigma(a,x)| |\langle f_n,\varphi_{a,x}
 \rangle_{\alpha,2}||\psi_{a,x}(y)| d\mu_{\alpha}(a,x) d\mu_{\alpha}(y) \nonumber\\
     &\leq&  C \|\varphi\|_{\alpha,\infty}\int_{\X} |\sigma(a,x)|\int_{\mathbb{R}_{+}^{d+1}}|\tau_x^\alpha\psi_{a}(y)| d\mu_{\alpha}(y) d\mu_{\alpha}(a,x) \nonumber\\
     &\leq&  C \|\varphi\|_{\alpha,\infty}\int_{\X} |\sigma(a,x)|\int_{\mathbb{R}_{+}^{d+1}}|\psi_{a}(y)| d\mu_{\alpha}(y) d\mu_{\alpha}(a,x) \nonumber\\
     &\leq& C \|\varphi\|_{\alpha,\infty}   \|\psi\|_{\alpha,1}\|\sigma\|_{L^{1}_{\alpha}(\X)}<\infty.
  \end{eqnarray}
Thus, according to the relations (\ref{3.9})$-$(\ref{3.12}) and the Lebesgue dominated convergence theorem
we deduce that
\begin{equation*}
    \lim_{n\to \infty} \|\mathcal{L}_{\varphi,\psi}(\sigma)(f_n)\|_{\alpha,1}=0
  \end{equation*}
  which completes the proof.
\end{proof}

\begin{thm}
   Under the same assumptions of Theorem \ref{thm3.3}, the Weinstein two-wavelet localization operator
  $$\mathcal{L}_{\varphi,\psi}(\sigma):L^{p}_{\alpha}(\mathbb{R}^{d+1}_+)\longrightarrow L^{p}_{\alpha}(\mathbb{R}^{d+1}_+)$$
  is compact for all $p\in[1,\infty]$.
\end{thm}
\begin{proof}
  We only need to show that the result holds for $p=\infty$. in fact, the operator
  \begin{equation}\label{3.13}
   \mathcal{L}_{\varphi,\psi}(\sigma):L^{\infty}_{\alpha}(\mathbb{R}^{d+1}_+)\longrightarrow L^{\infty}_{\alpha}(\mathbb{R}^{d+1}_+),
  \end{equation}
    is the adjoint of the operator $$\mathcal{L}_{\psi,\varphi}(\overline{\sigma}):L^{1}_{\alpha}(\mathbb{R}^{d+1}_+)\longrightarrow L^{1}_{\alpha}(\mathbb{R}^{d+1}_+),$$
   which is compact by the previous Proposition. Therefore by the duality property, the operator given by (\ref{3.13}) is compact. Finally, by an interpolation of the compactness on $L^{1}_{\alpha}(\mathbb{R}^{d+1}_+)$ and on $L^{\infty}_{\alpha}(\mathbb{R}^{d+1}_+)$ like the one given on \cite[Theorem 2.9]{bennett1988interpolation}, the proof
is complete.
\end{proof}
In the following Theorem, we state a compactness result for the Weinstein two-wavelet localization operator  analogue of Theorem \ref{thm3.7}.
\begin{thm}
   Under the same assumptions of Theorem \ref{thm3.7}, the Weinstein two-wavelet localization operator
  $$\mathcal{L}_{\varphi,\psi}(\sigma):L^{p}_{\alpha}(\mathbb{R}^{d+1}_+)\longrightarrow L^{p}_{\alpha}(\mathbb{R}^{d+1}_+)$$
  is compact for all $p\in[r,r']$.
\end{thm}
\begin{proof}
By the same manner in the proof of the pervious Theorem, the result is an immediate consequence of an interpolation of
\cite[Corollary 4.2]{mejjaoli2017new} and Proposition \ref{prop3.15} like the one given on \cite[Theorem 2.9]{bennett1988interpolation}.
\end{proof}
Using similar ideas as above we can prove the following reult.
\begin{thm}
   Under the same assumptions of Theorem \ref{thm6}, the Weinstein two-wavelet localization operator
  $$\mathcal{L}_{\varphi,\psi}(\sigma):L^{p}_{\alpha}(\mathbb{R}^{d+1}_+)\longrightarrow L^{p}_{\alpha}(\mathbb{R}^{d+1}_+)$$
  is compact for all $p\in[1,\infty]$.
\end{thm}
\subsection{Examples}
In this subsection, as in the paper of Wong \cite{wong2003localization}, we study some typical examples of the Weinstein two-wavelet localization operators.
We show that on the space $$\X=\left\{(a,x): x\in \mathbb{R}^{d+1}_+ \;\text{and}\; a>0\right\}$$ the  localization operators associated to admissible Weinstein wavelets $\varphi$ and $\psi$ and separable symbols $\sigma$ are paracommutators and we show that if the symbol is a function of $x$ only, then the localization operator can be expressed in terms of a paraproduct. We show in the end  if the symbol is a function of $a$ only, then the localization operator $\mathcal{L}_{\varphi,\psi}(\sigma)$  is a Weinstein multiplier.
\subsubsection{Paracomutators} Let $\sigma$ be a separable function on $\X$ given by
\begin{equation*}
  \forall (a,x)\in\X, \quad \sigma(a,x)=\chi(a)\zeta(x),
\end{equation*}
where $\chi$ and $\zeta$ are suitable functions,respectively, on $(0,\infty)$ and $\mathbb{R}^{d+1}_+$. Then, according to
Parseval's formula (\ref{MM}) and Fubini's theorem, we have for all $f,g\in L^2_\alpha(\mathbb{R}^{d+1}_+)$
\begin{eqnarray*}
  \langle \mathcal{L}_{\varphi, \psi}(\sigma)(f),\ g \rangle_{\alpha,2}&=& \int_{\X}
 \sigma(a,x)\Phi^W_\varphi(f)(a,x)\overline{\Phi^W_\psi(g)(a,x)}d\mu_\alpha(a,x)\\
   &=& \int_{0}^{\infty}\chi(a)\int_{\mathbb{R}^{d+1}_+}\int_{\mathbb{R}^{d+1}_+}\tau_\eta^\alpha\mathcal{F}_{W}(\zeta)(-\xi) \mathcal{F}_{W}(f)(\xi)\mathcal{F}_{W}(\overline{\check{\varphi}})(a\xi)  \\
  &&  \times \overline{\mathcal{F}_{W}(g)(\eta)\mathcal{F}_{W}(\overline{\check{\psi}})(a\xi)}d\mu_\alpha(\eta)d\mu_\alpha(\xi)\frac{da}{a} \\
   &=&  \int_{\mathbb{R}^{d+1}_+}\int_{\mathbb{R}^{d+1}_+}K(\xi,\eta)\tau_\eta^\alpha\mathcal{F}_{W}(\zeta)(-\xi) \mathcal{F}_{W}(f)(\xi)
   \overline{\mathcal{F}_{W}(g)(\eta)}\\
   && d\mu_\alpha(\eta)d\mu_\alpha(\xi),
\end{eqnarray*}
where
\begin{equation*}
  K(\xi,\eta)=\int_{0}^{\infty}\chi(a)\overline{\mathcal{F}_{W}(\varphi)(a\xi)}\mathcal{F}_{W}(\psi)(a\eta)\frac{da}{a},\quad \forall \xi,\eta\in \mathbb{R}^{d+1}_+.
\end{equation*}
Thus, the localization operator $\mathcal{L}_{\varphi,\psi}(\sigma)$ is a paracommutator with Weinstein kernel
$K$ and symbol $\zeta$.
\subsubsection{Paraproduct} In this example, we specialize to the case when the symbol $\sigma$ is
a function of $x$ only, i.e.
\begin{equation*}
  \forall (a,x)\in\X, \quad \sigma(a,x)=\zeta(x),
\end{equation*}
where $\zeta$ is a suitable function on $\mathbb{R}^{d+1}_+$.  Indeed, using Parseval's formula (\ref{MM}) and Fubini's theorem as in the preceding example, we get for all  $f,g\in L^2_\alpha(\mathbb{R}^{d+1}_+)$
\begin{eqnarray}\label{ex3.12}
 \langle \mathcal{L}_{\varphi, \psi}(\sigma)(f),\ g \rangle_{\alpha,2}&=& \int_{\X}
 \sigma(a,x)\Phi^W_\varphi(f)(a,x)\overline{\Phi^W_\psi(g)(a,x)}d\mu_\alpha(a,x) \nonumber \\
   &=&\int_{\mathbb{R}^{d+1}_+}\left[\int_{0}^{\infty}(\check{\zeta}(\Theta_a*f)*\psi_a)(x)\frac{da}{a}\right]\overline{g(x)}
   d\mu_\alpha(x),
\end{eqnarray}
where $\Theta(x)=\overline{\varphi(-x)}$. Therefore, we deduce that
\begin{equation*}
  \mathcal{L}_{\varphi, \psi}(\sigma)(f)(x)= \int_{0}^{\infty}(\check{\zeta}(\Theta_a*f)*\psi_a)(x)\frac{da}{a},\quad\forall x\in \mathbb{R}^{d+1}_+.
\end{equation*}
 This formula for the Weinstein two-wavelet localization operator  is an interesting formula in its own right. Further analysis of (\ref{ex3.12}) using Fubini's theorem gives
\begin{equation}\label{Lpfipsi}
   \langle \mathcal{L}_{\varphi, \psi}(\sigma)(f),\ g \rangle_{\alpha,2}=\int_{\mathbb{R}^{d+1}_+}\check{\zeta}(x)p_{\varphi,\psi} (f,g)(x)d\mu_\alpha(x),
\end{equation}
where
\begin{equation}
  p_{\varphi,\psi} (f,g)(x)=\int_{0}^{\infty}(\Theta_a*f)\overline{(\Upsilon_a*g)(x)}\frac{da}{a},\quad\forall x\in\mathbb{R}^{d+1}_+,
\end{equation}
with $\Upsilon(x)=\overline{\psi(-x)}$.

Several versions of paraproducts exist in the literature. It should be remarked that the notion of a paraproduct is rooted in Bony's work \cite{bony1981calcul} on linearization of nonlinear problems.
The paraproduct connection (\ref{ex3.12}) and the fact that the Weinstein two-wavelet localization operators  associated to symbols $\sigma$ in $L^\infty_\alpha(\X)$ are bounded linear operators (see \cite[Corollary 3.2]{mejjaoli2017new}) such that
\begin{equation*}
  \|\mathcal{L}_{\varphi, \psi}(\sigma)\|_{S_\infty}\leq \sqrt{C_\varphi C_\psi}  \|\sigma\|_{L^\infty_\alpha(\X)}.
\end{equation*}
allow us to give an $L^1_\alpha$-estimate on the paraproduct $p_{\varphi,\psi} (f,g)$, where $\varphi$ and $\psi$ are the
Weinstein wavelets, and $f$ and $g$ are functions in $ L^2_\alpha(\mathbb{R}^{d+1}_+)$.  First, we need the following Lemma.
\begin{lem}
  Let  $\varphi$ and $\psi$ two Weinstein wavelets such that $(\varphi,\psi)$ be a Weinstein two-wavelet. Then we have for all $f$ and $g$  in  $L^2_\alpha(\mathbb{R}^{d+1}_+)$
\begin{equation*}
 \int_{\mathbb{R}^{d+1}_+}p_{\varphi,\psi} (f,g)(x)d\mu_\alpha(x)=  C_{\varphi,\psi}\langle f,g \rangle_{\alpha,2},
\end{equation*}
\end{lem}
\begin{proof}
According to relation (\ref{F(f*g)}), Parseval's formula (\ref{MM}) and Fubini's theorem, we obtain
\begin{eqnarray*}
   \int_{\mathbb{R}^{d+1}_+} &&\!\!\!\!\!\!\!\!\!\!\!\!\! p_{\varphi,\psi} (f,g)(x)\mu_\alpha(x)= \int_{\mathbb{R}^{d+1}_+}\int_{0}^{\infty}(\Theta_a*f)(x)\overline{(\Upsilon_a*g)(x)}\frac{da}{a}d\mu_\alpha(x) \\
   &=& \int_{0}^{\infty}\left[\int_{\mathbb{R}^{d+1}_+} \mathcal{F}_{W}(f)(\xi)\mathcal{F}_{W}(\Theta)(a\xi)
   \overline{ \mathcal{F}_{W}(g)(\xi)\mathcal{F}_{W}(\Upsilon)(a\xi)}d\mu_\alpha(x)\right]\frac{da}{a}\\
   &=& \left[\int_{0}^{\infty}\mathcal{F}_{W}(\Theta)(a\xi)\overline{\mathcal{F}_{W}(\Upsilon)(a\xi)}\frac{da}{a}\right] \mathcal{F}_{W}(f)(\xi) \overline{\mathcal{F}_{W}(g)(\xi)}d\mu_\alpha(x)\\
   &=&  C_{\varphi,\psi}\langle f,g \rangle_{\alpha,2}.
\end{eqnarray*}
\end{proof}
An $L^1_\alpha$-estimate for the paraproduct $p_{\varphi,\psi} (f,g)$ in terms of the $L^2_\alpha$-norms of $f$ and $g$ is given in the following theorem.
\begin{thm}
  Let  $\varphi$ and $\psi$ two Weinstein wavelets such that
  \begin{equation*}
    \|\varphi\|_{\alpha,2}=\|\psi\|_{\alpha,2}=1.
  \end{equation*}
   Then we have for all $f$ and $g$  in $L^2_\alpha(\mathbb{R}^{d+1}_+)$
   \begin{equation*}
      \|p_{\varphi,\psi} (f,g)\|_{\alpha,1}\leq \sqrt{C_\varphi C_\psi}  \|f\|_{\alpha,2} \|g\|_{\alpha,2}.
   \end{equation*}
\end{thm}
\begin{proof}
  From \cite[Corollary 3.2]{mejjaoli2017new}, we know that
\begin{equation*}
  \|\mathcal{L}_{\varphi, \psi}(\sigma)\|_{S_\infty}\leq \sqrt{C_\varphi C_\psi}  \|\sigma\|_{L^\infty_\alpha(\X)}.
\end{equation*}
\end{proof}
Therefore, by relation (\ref{Lpfipsi}) and Cauchy-Schwarz inequality, we get
\begin{equation*}
  \frac{1}{\sqrt{C_\varphi C_\psi}}\left|\int_{\mathbb{R}^{d+1}_+}\check{\zeta}(x)p_{\varphi,\psi} (f,g)(x)d\mu_\alpha(x)\right| \leq \|\zeta\|_{\alpha,\infty} \|f\|_{\alpha,2} \|g\|_{\alpha,2}.
\end{equation*}
Since $\frac{1}{\sqrt{C_\varphi C_\psi}}p_{\varphi,\psi} (f,g)$ belongs to $L^1_\alpha(\mathbb{R}^{d+1}_+)$, it follows from the Hahn–Banach theorem that is in the dual $\left(L^\infty_\alpha(\mathbb{R}^{d+1}_+)\right)^*$ of $L^\infty_\alpha(\mathbb{R}^{d+1}_+)$ and we have

\begin{equation*}
  \frac{1}{\sqrt{C_\varphi C_\psi}}\|p_{\varphi,\psi}(f,g)\|_{\alpha,1}\leq  \|f\|_{\alpha,2} \|g\|_{\alpha,2}.
\end{equation*}
\subsubsection{Weinstein Multipliers}
We see in this section that if the symbol $\sigma$ is a function of $a$ only, then the Weinstein two-wavelet localization operator
 is a Weinstein multiplier $\mathcal{T}^W_m$ with symbol $m$ defined on $L^2_\alpha(\mathbb{R}^{d+1}_+)$ as follow
\begin{equation*}
  \mathcal{T}^W_m= \mathcal{F}_{W}^{-1}(m \mathcal{F}_{W}(f))
\end{equation*}
\begin{prop}
  Let $\sigma$ be a function on $\X$ given by
  \begin{equation*}
    \sigma(a,x)=\chi(a), \quad\forall (a,x)\in \X,
  \end{equation*}
  where $\chi$ is a suitable function on $(0,\infty).$ Then, we have
  \begin{equation*}
    \mathcal{L}_{\varphi, \psi}(\sigma)=\mathcal{T}^W_m,
  \end{equation*}
where $\mathcal{T}^W_m$ is the Weinstein multiplier with symbol $m$ given by
 \begin{equation*}
   m(\xi)=\int_{0}^{\infty}\chi(a)\overline{\mathcal{F}_{W}(\varphi)(a\xi)}\mathcal{F}_{W}(\psi)(a\xi)\frac{da}{a},\quad \forall \xi\in \mathbb{R}^{d+1}_+.
 \end{equation*}
 \end{prop}
  \begin{proof}
    For any positive integer $m$, we define $I_m$ by

      \begin{eqnarray*}
      I_m &=&  \int_{\mathbb{R}^{d+1}_+}\int_{\mathbb{R}^{d+1}_+}\int_{\X}e^{-\frac{\|b\|^2}{2m^2}}\chi(a)\overline{\mathcal{F}_{W}(\varphi)
      (a\xi)}\mathcal{F}_{W}(\psi)(a\eta) \Lambda_{\alpha}^d(ib,\xi)\Lambda_{\alpha}^d(-ib,\eta)\\
    && \times\mathcal{F}_{W}(f)(\xi) \overline{\mathcal{F}_{W}(g)(\eta)}d\mu_\alpha(b)d\mu_\alpha(\eta)d\mu_\alpha(\xi) \frac{da}{a}.
      \end{eqnarray*}
 Then using Fubini's theorem and the fact that the Weinstein transform of the function $h(x)=e^{-\frac{\|x\|^2}{2}}$  is equal to itself for all $x\in \mathbb{R}^{d+1}_+$ (see \cite[Example 1]{gasmi2016inversion}), we get
     \begin{eqnarray}\label{suiteinteg}
       I_m &=& \int_{\mathbb{R}^{d+1}_+}\int_{\mathbb{R}^{d+1}_+}\int_{0}^\infty \tau_{-\eta}^\alpha h_m(\xi) \chi(a)\overline{\mathcal{F}_{W}(\varphi)(a\xi)} \mathcal{F}_{W}(\psi)(a\eta) \nonumber\\
        && \times \mathcal{F}_{W}(f)(\xi) \overline{\mathcal{F}_{W}(g)(\eta)} d\mu_\alpha(\eta)d\mu_\alpha(\xi) \frac{da}{a}  \nonumber\\
        &=& \int_{\mathbb{R}^{d+1}_+}\int_{0}^\infty \chi(a)\overline{\mathcal{F}_{W}(\varphi)(a\xi)} \left(\mathcal{F}_{W}(\psi_a) \overline{\mathcal{F}_{W}(g)}* h_m\right)(\xi)
       \mathcal{F}_{W}(f)(\xi) d\mu_\alpha(\xi) \frac{da}{a}.\nonumber\\
     \end{eqnarray}
On the other hand, it's easy to see that
\begin{equation}\label{convF}
   \mathcal{F}_{W}(\psi_a) \overline{\mathcal{F}_{W}(g)}* h_m\rightarrow \mathcal{F}_{W}(\psi_a) \overline{\mathcal{F}_{W}(g)}
\end{equation}
in $L^2_\alpha(\mathbb{R}^{d+1}_+)$ and almost everywhere on $\mathbb{R}^{d+1}_+$ as $n\rightarrow\infty$. Therefore, according to relations (\ref{suiteinteg}) and (\ref{convF}), we obtain
\begin{equation}\label{convIm}
  I_m\rightarrow \int_{\mathbb{R}^{d+1}_+}\int_{0}^\infty \chi(a)\overline{\mathcal{F}_{W}(\varphi)(a\xi)} \mathcal{F}_{W}(\psi) (a\xi)\overline{\mathcal{F}_{W}(g)(\xi)}
       \mathcal{F}_{W}(f)(\xi) d\mu_\alpha(\xi) \frac{da}{a},
\end{equation}
 as $n\rightarrow\infty$. Subsequently, using Lebesgue's dominated convergence theorem, we see that
\begin{equation*}
  I_m\rightarrow \langle \mathcal{L}_{\varphi, \psi}(\sigma)(f),\ g \rangle_{\alpha,2}.
\end{equation*}
Therefore,
\begin{equation*}
\langle \mathcal{L}_{\varphi, \psi}(\sigma)(f),\ g \rangle_{\alpha,2}=\langle \mathcal{T}_{m}^Wf,\ g \rangle_{\alpha,2},
\end{equation*}
for all $f$ and $g$ in $L^2_\alpha(\mathbb{R}^{d+1}_+)$,  where $\mathcal{T}^W_m$ is the Weinstein multiplier with symbol $m$ given by
 \begin{equation*}
   m(\xi)=\int_{0}^{\infty}\chi(a)\overline{\mathcal{F}_{W}(\varphi)(a\xi)}\mathcal{F}_{W}(\psi)(a\xi)\frac{da}{a},\quad \forall \xi\in \mathbb{R}^{d+1}_+.
 \end{equation*}
 \end{proof}

\end{document}